\documentclass[11pt]{article}
\usepackage[a4paper,margin=30mm]{geometry}
\usepackage{amsmath,amssymb,amsthm,mathtools}
\usepackage[hidelinks]{hyperref}
\usepackage{microtype}
\usepackage{tikz}
\usepackage{pgfplots}
\pgfplotsset{compat=1.18}
\usetikzlibrary{patterns}

\newtheorem{theorem}{Theorem}
\newtheorem{proposition}[theorem]{Proposition}
\newtheorem{lemma}[theorem]{Lemma}
\newtheorem{corollary}[theorem]{Corollary}
\theoremstyle{definition}
\newtheorem{definition}[theorem]{Definition}
\newtheorem{remark}[theorem]{Remark}

\newcommand{\R}{\mathbb{R}}
\newcommand{\Z}{\mathbb{Z}}
\newcommand{\Sone}{\mathbb{S}^{1}}
\newcommand{\conv}{\operatorname{conv}}
\newcommand{\CR}{\operatorname{CR}}
\newcommand{\Var}{\operatorname{Var}}

\title{A New Unconditional Lower Bound for Shoreline Search}
\author{Alexander Temerev\\ University of Geneva}
\date{August 2026}

\begin{document}
\maketitle

\begin{abstract}
A unit-speed searcher starts at the origin of the Euclidean plane and must hit
an unknown straight line.  Neither the direction of the line nor its distance
from the origin is known.  We prove that every deterministic search path has
competitive ratio at least
\[
  C_{\log}\approx12.5937096701246675.
\]
The bound is unconditional within the class of deterministic paths: the path
need not be cyclic, self-similar, spiral-like, or monotone in angle.  The proof
compares each one-dimensional
projection of the path with a zigzag whose alternating record turns have been
sorted, reads the resulting completion constraints as jobs with
scale-dependent deadlines, and lower-bounds their activity in logarithmic
time.  Averaging over all projection directions then meets the exact Euclidean
velocity budget.  Specialised to one dimension the same argument recovers the
cow-path constant~$9$ exactly.  Arb ball arithmetic encloses the single
numerical constant used in the theorem.
\end{abstract}

\noindent\textbf{MSC 2020:} 90B40, 68W27, 52A10.\quad
\textbf{Keywords:} shoreline search, online search, competitive ratio,
logarithmic spiral, linear search, lower bound.

\section{Introduction}

In shoreline search, a unit-speed searcher begins at a known point while a
straight target line has unknown direction and unknown distance.  If the line
is at distance $d$, an omniscient searcher pays $d$; an online path is judged
by the largest ratio between its first-hitting time and $d$.

The best known path is a logarithmic spiral with competitive ratio
$13.8111\ldots$ \cite{baezayates1988,baezayates1993,finchzhu2005}; its
optimality among all paths is open.  The strongest unconditional lower bound
is $6.3972\ldots$, which is Isbell's optimal constant for the easier problem
in which the distance of the line is known in advance \cite{isbell1957}.  The
value $12.5385\ldots$ is known only under a cyclic (spiral-like) restriction
on the path, where it is inherited from the axis-parallel shoreline problem
\cite{langetepe2012}; see
\cite{acharjee2020,dobrev2020,georgiou2026} for recent accounts and related
lower-bound work, and \cite{gal1980,alperngal2003} for the search-game
background.
Our result is the following.

\begin{theorem}[unconditional lower bound]\label{thm:main}
For every $d_0>0$, every deterministic path which searches all straight lines
at distances $d\geq d_0$ satisfies
\[
  \CR_{d_0}\geq C_{\log}>12.5937.
\]
Here $C_{\log}$ is the unique constant satisfying
$G(C_{\log})=2/\pi$, for the function $G$ defined in~\eqref{eq:Gdef};
numerically $C_{\log}\approx12.5937096701246675$, and
Section~\ref{sec:numerics} certifies
\[
 12.59370967012466<C_{\log}<12.59370967012468.
\]
No cyclicity or other ordering assumption is imposed on the planar path.
\end{theorem}

The word ``unconditional'' means that no structural restriction is placed on
the deterministic path; randomized search is not considered here.  For every
fixed direction we shall compare the scalar projection of the path with a
sorted zigzag.  These sorted zigzags are comparison objects and may differ
from one direction to another; we never replace the planar path by one common
cyclic path.

The proof has four steps.  Competitive search gives a disk-completion deadline
at every scale.  In each scalar projection, the completion variation is
bounded below by that of the zigzag whose alternating record turns are sorted
by magnitude.  The sorted record turns become jobs which must be processed
before geometric deadlines.  A finite-window scheduling lemma lower-bounds
their logarithmic-time activity.  Finally, the exact identity
\[
  \int_0^\pi |\langle (\cos\theta,\sin\theta),v\rangle|\,d\theta
  =2\|v\|
\]
adds those directional lower bounds using precisely the available Euclidean
speed.  Section~\ref{sec:checks} records three consistency checks: the
argument reproduces the cow-path constant $9$ exactly in one dimension, gives
a value below Langetepe's optimal cyclic constant for the axis-parallel
problem, and is satisfied with room to spare by the logarithmic spiral.

\section{The search problem and disk deadlines}

Let $\gamma:[0,\infty)\to\R^2$ be a continuous, locally rectifiable path with
$\gamma(0)=0$.  A path that searches all the prescribed lines has infinite
length and admits an arclength parametrization.  Any admissible
$1$-Lipschitz presentation with stops reaches each point no earlier than its
arclength factorization, so it can only have a larger competitive ratio.  We
may therefore take $\gamma$ to be locally absolutely continuous and
parametrized by arclength, with $\|\gamma'(t)\|=1$ almost everywhere.  For
$u\in\Sone$ and $d>0$, put
\[
 L(u,d)=\{x\in\R^2:\langle u,x\rangle=d\},
 \qquad
 T_\gamma(u,d)=\inf\{t\geq0:\gamma(t)\in L(u,d)\}.
\]
For a fixed lower distance cutoff $d_0>0$, define
\[
 \CR_{d_0}(\gamma)
 =\sup_{u\in\Sone,\ d\geq d_0}\frac{T_\gamma(u,d)}d.
\]
This is the standard nondegenerate formulation; the value of $d_0$ only fixes
an initial length scale.

Write
\[
 K_t=\conv\gamma([0,t]),
 \qquad
 h_t(u)=\max_{0\leq s\leq t}\langle u,\gamma(s)\rangle.
\]

\begin{proposition}[disk deadline]\label{prop:diskdeadline}
If $\CR_{d_0}(\gamma)\leq C$, then for every $t\geq Cd_0$,
\[
  \frac tC B_2\subseteq K_t.
\]
Equivalently, for every $u\in\Sone$, the scalar projection
$x_u(s)=\langle u,\gamma(s)\rangle$ has visited both $t/C$ and $-t/C$ by
time $t$.
\end{proposition}

\begin{proof}
Set $d=t/C\geq d_0$.  By $C$-competitiveness, every line $L(u,d)$ has been
hit by time $Cd=t$.  Continuity and $x_u(0)=0$ say that this is equivalent to
$h_t(u)\geq d$.  Since the same statement holds for every $u$, the
support-function criterion for convex inclusion gives $dB_2\subseteq K_t$.
Applying the support inequalities to $u$ and $-u$ gives the scalar statement.
\end{proof}

For a fixed $u$, let
\[
 V_u(t)=\int_0^t|\langle u,\gamma'(s)\rangle|\,ds
\]
be the total variation of the projection $x_u$ on $[0,t]$.  Then $V_u$ is
absolutely continuous, $0\leq V_u'\leq1$ almost everywhere, and
$V_u(t)\leq t$.  Proposition~\ref{prop:diskdeadline} says that $x_u$ has
visited both $r$ and $-r$ by time $Cr$ for every $r\geq d_0$; the next
section turns this into a lower bound for $V_u$.

Fixing any direction already contains the classical cow-path problem, so
every competitive constant satisfies
\begin{equation}\label{eq:cowpath}
  C\geq9,
\end{equation}
also when target distances are restricted to $d\geq d_0$
\cite{beck1964,becknewman1970}.  We shall not rely on
\eqref{eq:cowpath}: the elementary bound \eqref{eq:factorbound} below already
gives $C\geq5$, which is all that the argument needs, and
Section~\ref{sec:checks} shows that the argument reproduces
\eqref{eq:cowpath} on its own.

\section{Alternating records and a scalar normal form}\label{sec:normalform}

Throughout this section $x:[0,\infty)\to\R$ is continuous, of bounded
variation on compact intervals, $x(0)=0$, and $x$ reaches both $r$ and $-r$
in finite time for every $r\geq d_0$; by continuity it then reaches every
level in $[-r,r]$.  In the application, $x=x_u$ is a projection of $\gamma$
in ambient time.  Write
\[
 M(t)=\max_{s\leq t}x(s),\qquad N(t)=\max_{s\leq t}\bigl(-x(s)\bigr),
\]
\[
 \tau(r)=\inf\{t: M(t)\geq r\text{ and }N(t)\geq r\},
 \qquad
 W(r)=\Var\bigl(x;[0,\tau(r)]\bigr),
\]
the \emph{completion time} and \emph{completion variation} of radius $r$.
For $y>0$ let $\sigma^{+}(y)$ and $\sigma^{-}(y)$ be the first passage times
of the levels $y$ and $-y$; both are strictly increasing in $y$, and
$\tau(r)=\max\{\sigma^{+}(r),\sigma^{-}(r)\}$.  For $x=x_u$,
Proposition~\ref{prop:diskdeadline} gives $\tau(r)\leq Cr$, hence
\begin{equation}\label{eq:completionbound}
  W(r)=\Var\bigl(x_u;[0,\tau(r)]\bigr)\leq\Var\bigl(x_u;[0,Cr]\bigr)
  =V_u(Cr)\qquad(r\geq d_0).
\end{equation}

The idea of this section is simple.  Follow the running maximum $M$ and the
running minimum $-N$: the path alternately extends one of them, and each
switch from one side to the other is a \emph{turn}.  Between consecutive
turns of magnitudes $m_{k-1}$ and $m_k$ the path crosses zero and spends
variation at least $m_{k-1}+m_k$.  Sorting the turn magnitudes gives the
cheapest conceivable schedule for completing every radius, and
Lemma~\ref{lem:normalform} shows that the actual path is never cheaper.  The
only technical point is that a general continuous path may have infinitely
many small turns accumulating at its departure from the origin; the
definition below and Lemma~\ref{lem:records}, proved in
Appendix~\ref{app:turns}, take care of this.  Figure~\ref{fig:turns} illustrates the construction.

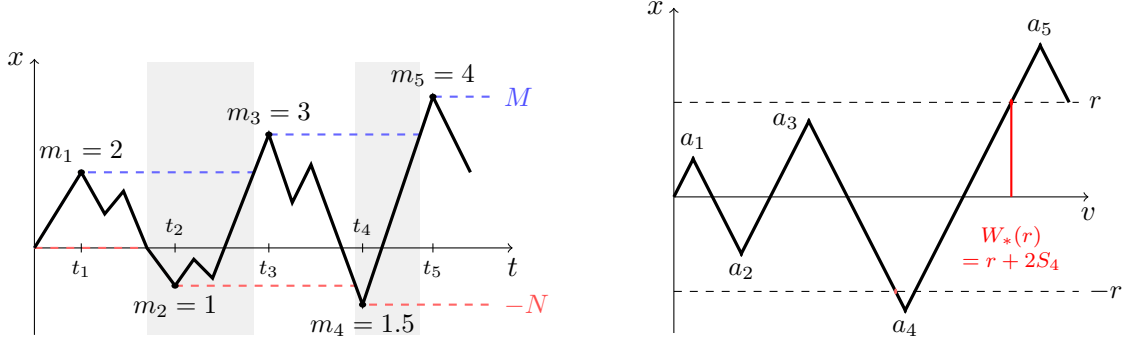
\begin{figure}[t]
\centering
\begin{tikzpicture}[xscale=0.62,yscale=0.5]
 \fill[gray!12] (2.4,-2.3) rectangle (4.68,4.9);
 \fill[gray!12] (6.85,-2.3) rectangle (8.23,4.9);
 \draw[->] (0,0) -- (10.2,0) node[below] {$t$};
 \draw[->] (0,-2.3) -- (0,5.0) node[left] {$x$};
 \draw[dashed,thick,blue!60] (0,0) -- (1,2) -- (4.68,2) -- (5,3) -- (8.23,3) -- (8.5,4) -- (9.8,4);
 \draw[dashed,thick,red!60] (0,0) -- (2.4,0) -- (3,-1) -- (6.85,-1) -- (7,-1.5) -- (9.8,-1.5);
 \node[blue!70,right,font=\small] at (9.8,4) {$M$};
 \node[red!70,right,font=\small] at (9.8,-1.5) {$-N$};
 \draw[very thick] (0,0) -- (1,2) -- (1.5,0.9) -- (1.9,1.5) -- (2.4,0) -- (3,-1)
   -- (3.4,-0.3) -- (3.8,-0.8) -- (5,3) -- (5.5,1.2) -- (5.9,2.2) -- (7,-1.5)
   -- (8.5,4) -- (9.3,2);
 \foreach \x/\y/\lab/\pos in {1/2/{m_1=2}/above, 3/-1/{m_2=1}/below, 5/3/{m_3=3}/above,
                              7/-1.5/{m_4=1.5}/below, 8.5/4/{m_5=4}/above}
   { \fill (\x,\y) circle (2.2pt); \node[\pos,font=\small] at (\x,\y) {$\lab$}; }
 \foreach \x/\lab/\pos in {1/t_1/below,3/t_2/above,5/t_3/below,7/t_4/above,8.5/t_5/below}
   { \draw (\x,0.12) -- (\x,-0.12); \node[\pos=1pt,font=\scriptsize] at (\x,0) {$\lab$}; }
\end{tikzpicture}
\hfill
\begin{tikzpicture}[xscale=0.255,yscale=0.5]
 \draw[->] (0,0) -- (21.5,0) node[below] {$v$};
 \draw[->] (0,-3.6) -- (0,5.0) node[left] {$x$};
 \draw[very thick] (0,0) -- (1,1) -- (3.5,-1.5) -- (7,2) -- (12,-3) -- (19,4) -- (20.5,2.5);
 \foreach \x/\y/\lab/\pos in {1/1/{a_1}/above, 3.5/-1.5/{a_2}/below, 7/2/{a_3}/{left}, 12/-3/{a_4}/below, 19/4/{a_5}/above}
   { \fill (\x,\y) circle (2.2pt); \node[\pos,font=\small] at (\x,\y) {$\lab$}; }
 \draw[dashed] (0,2.5) -- (21,2.5) node[right,font=\small] {$r$};
 \draw[dashed] (0,-2.5) -- (21,-2.5) node[right,font=\small] {$-r$};
 \fill[red!60] (11.5,-2.5) circle (2.2pt);
 \fill[red] (17.5,2.5) circle (2.6pt);
 \draw[red,thick] (17.5,2.5) -- (17.5,0);
 \node[red,anchor=north,align=center,font=\scriptsize] at (17.5,-0.5) {$W_*(r)$\\$=r+2S_4$};
\end{tikzpicture}
\caption{Left: a scalar path with its running maximum $M$ and minimum $-N$
(dashed).  A phase is a maximal stretch during which only one side sets
records; the negative phases are shaded (they end when the path first
exceeds the old positive record).  The turns $t_k$ are the last record times
of the phases; the wiggles between them set no records and are ignored.
Right: the sorted zigzag, with $a_1\leq a_2\leq\cdots$ the sorted turn
magnitudes ($1,1.5,2,3,4$), plotted against variation $v$.  For
$a_3<r\leq a_4$ both levels $\pm r$ have been reached after variation
$r+2S_4$; Lemma~\ref{lem:normalform} says that the original path needs at
least that much.}
\label{fig:turns}
\end{figure}

\begin{definition}[record times, phases, turns]\label{def:turns}
A time $t>0$ is a \emph{positive record time} if $x(t)=M(t)>M(s)$ for all
$s<t$, i.e.\ $t=\sigma^{+}(x(t))$; \emph{negative record times} are defined
with $-x$ and $N$.  Let $P$ and $Q$ be the sets of positive and negative
record times; they are disjoint, since $x>0$ on $P$ and $x<0$ on $Q$.  Let
$t_{\mathrm{dep}}=\sup\{t:x=0\text{ on }[0,t]\}$ be the departure time.  For
$t>t_{\mathrm{dep}}$ let $\rho(t)$ be the last record time in
$(t_{\mathrm{dep}},t]$ (it exists by Lemma~\ref{lem:records}(i)), and call
$t$ \emph{positive} or \emph{negative} according to the sign of $\rho(t)$.
A \emph{phase} is a maximal interval of $(t_{\mathrm{dep}},\infty)$ on
which the sign is constant.  These maximal intervals are pairwise disjoint
and partition $(t_{\mathrm{dep}},\infty)$, because every time in that set has
exactly one sign.  The \emph{turn} of a phase is its last record time $t_k$,
write $\varepsilon_k\in\{+,-\}$ for the sign of that phase, and let
$m_k=|x(t_k)|$ be the \emph{turn magnitude}.  In particular, $t_k$ is the
first passage time of the signed level $\varepsilon_km_k$.
\end{definition}

\begin{lemma}[turns]\label{lem:records}
Under the standing assumptions of this section:
\begin{enumerate}
\item[(i)] For every $t>t_{\mathrm{dep}}$ the set of record times in
$(t_{\mathrm{dep}},t]$ has a largest element $\rho(t)$, which lies in the
phase of $t$.  Every phase contains record times and has a last record time.
\item[(ii)] Only finitely many phases meet a compact subinterval of
$(t_{\mathrm{dep}},\infty)$.  Hence the phases are indexed in temporal order
by an interval $K\subseteq\Z$, which is unbounded above; if $K$ is unbounded
below then $t_k\downarrow t_{\mathrm{dep}}$ as $k\to-\infty$.
\item[(iii)] Consecutive phases have opposite signs, so the signs
$\varepsilon_k$ of the turns alternate and $x(t_k)=\varepsilon_km_k$.
\item[(iv)] On each side the turn magnitudes are strictly increasing, and
$m_k\to\infty$ as $k\to\infty$.
\item[(v)] $\Var(x;[t_{k-1},t_k])\geq m_{k-1}+m_k$.
\item[(vi)] For $\varepsilon\in\{+,-\}$ and $r>0$, the first passage time
$\sigma^{\varepsilon}(r)$ is a record time of sign $\varepsilon$; if it lies
in phase $j$ then $\varepsilon_j=\varepsilon$, $m_j\geq r$,
$t_{j-1}<\sigma^{\varepsilon}(r)\leq t_j$, and every turn of sign
$\varepsilon$ with index $<j$ has magnitude $<r$.
\item[(vii)] For $0<\varepsilon_0<X$ only finitely many turns have magnitude
in $[\varepsilon_0,X]$; if $K$ is unbounded below then $m_k\to0$ as
$k\to-\infty$; and $\sum_{k:\,t_k\leq T}m_k\leq\Var(x;[0,T])$ for every $T$.
\end{enumerate}
\end{lemma}

By Lemma~\ref{lem:records}(vii) the multiset of turn magnitudes can be
enumerated in nondecreasing order,
\[
 (a_i)_{i\in I},\qquad a_i\leq a_{i+1},\qquad I\subseteq\Z
 \text{ an interval unbounded above},
\]
ties being broken arbitrarily, and the partial sums
$S_i=\sum_{j\in I,\,j\leq i}a_j$ are finite, because all turns of magnitude
$\leq a_i$ occur by time $\tau(a_i)$ and Lemma~\ref{lem:records}(vii) applies.
We call $(a_i)$ the
\emph{sorted record sequence} of $x$.%
\footnote{One might try instead to sort the
turns by adjacent exchanges $(-p,+A,-b,+c)\mapsto(-p,+b,-A,+c)$ with
$p<b<A<c$.  A single such exchange never increases $W$, but after one
exchange the block is no longer a record skeleton and the next inversion may
violate $p<b$ or $A<c$; for the skeleton $+10,-1,+11,-2,+12$ every exchange
sequence leading to the sorted order passes through an uncovered exchange,
and that exchange does increase $W$ (from $25$ to $51$ at $r=3$).  The
direct comparison of Lemma~\ref{lem:normalform} avoids exchanges altogether.}
When $I$ has a least element the
sorted zigzag $0\to+a_{1}\to-a_{2}\to+a_{3}\to\cdots$ is a scalar path
whose completion variation at a radius $r$ with $a_{i-1}<r\leq a_i$ equals
$r+2S_i$: every magnitude through $a_i$ is paid twice and the level $r$ once
on the following leg.  This explains the name of the following quantity, but
we shall only use it as a number.

\begin{lemma}[scalar normal form]\label{lem:normalform}
Let $(a_i)_{i\in I}$ be the sorted record sequence of $x$.  For $r>0$ let
$i=i(r)$ be the index with $a_i=\min\{a_j:a_j\geq r\}$ and $a_{i-1}<r$
(if $i-1\notin I$ the last condition is void).  Then
\begin{equation}\label{eq:scalarwork}
  W(r)\;\geq\;W_*(r):=r+2S_i .
\end{equation}
\end{lemma}

\begin{proof}
Let $\varepsilon$ be the sign whose first passage of $r$ is later, so that
$\tau(r)=\sigma^{\varepsilon}(r)$, and let $j$ be the phase containing
$\sigma^{\varepsilon}(r)$.  By Lemma~\ref{lem:records}(v),(vi), the
variation up to $\tau(r)$ consists of the legs ending at
$t_k$, $k<j$, and of the passage from $\varepsilon_{j-1}m_{j-1}=-\varepsilon
m_{j-1}$ at time $t_{j-1}$ to $\varepsilon r$ at time $\tau(r)$:
\[
 W(r)\;\geq\;\sum_{k<j}(m_{k-1}+m_k)+(m_{j-1}+r)
 \;=\;r+2\sum_{k<j}m_k ,
\]
with the convention $m_{\min K-1}=0$ if $K$ has a least element (if $K$ is
unbounded below the sum is the convergent series of
Lemma~\ref{lem:records}(vii)).  It remains to show
$\sum_{k<j}m_k\geq S_i$, and since
$S_i=\sum_{a_l<r}a_l+a_i$ it suffices to show that every turn of magnitude
$<r$ has index $<j$, and that some turn of index $<j$ has magnitude $\geq r$.

Turns of sign $\varepsilon$ and magnitude $<r$ have index $<j$ by
Lemma~\ref{lem:records}(vi), because turns of sign $\varepsilon$ with index
$\geq j$ have magnitude $\geq m_j\geq r$.  A turn of sign $-\varepsilon$ and
magnitude $m<r$ is the first passage time of the level $-\varepsilon m$, hence
occurs before $\sigma^{-\varepsilon}(r)\leq\tau(r)\leq t_j$; being of sign
$-\varepsilon$ it is not the turn of phase $j$, so its index is $<j$.
Finally $\sigma^{-\varepsilon}(r)<\sigma^{\varepsilon}(r)=\tau(r)$ lies in a
phase $j'$ of sign $-\varepsilon$ with $m_{j'}\geq r$ by
Lemma~\ref{lem:records}(vi), and $j'<j$ because phases are ordered in time.
Then $m_{j'}\geq\min\{a_l:a_l\geq r\}=a_i$, and
$\sum_{k<j}m_k\geq\sum_{a_l<r}a_l+m_{j'}\geq S_i$.
\end{proof}

We now apply Lemma~\ref{lem:normalform} to $x=x_u$, for which
\eqref{eq:completionbound} holds.  For every index $i\in I$
with $a_{i-1}\geq d_0$ define a job
\begin{equation}\label{eq:jobs}
  D_i=Ca_{i-1},
  \qquad J_i=2a_i .
\end{equation}
Here $D_i$ is an ambient deadline and $J_i$ is required projected variation.

\begin{lemma}[workload]\label{lem:workload}
For $t\geq Cd_0$, let $i^*=i^*(t)$ be the least index with
$a_{i^*}>t/C$.  Then
\begin{equation}\label{eq:workload}
 V_u(t)\geq\frac tC+2S_{i^*}
 \geq\frac tC+\sum_{D_i\leq t}J_i ,
\end{equation}
where only the jobs \eqref{eq:jobs} are included.
\end{lemma}

\begin{proof}
Put $r=t/C\geq d_0$.  The index $i^*$ exists because the $a_l$ are locally
finite and unbounded.  For $r'>r$
close to $r$ we have $\{l:a_l<r'\}=\{l:a_l\leq r\}$ and
$\min\{a_l:a_l\geq r'\}=a_{i^*}$, hence $i(r')=i^*$, and
\eqref{eq:completionbound} with \eqref{eq:scalarwork} gives
$V_u(Cr')\geq r'+2S_{i^*}$.  Letting $r'\downarrow r$ and using the
continuity of $V_u$,
\[
 V_u(t)\geq\frac tC+2S_{i^*}.
\]
A job $l$ satisfies $D_l\leq t$ iff $a_{l-1}\leq r$ iff $l-1<i^*$, so
$\sum_{D_l\leq t}J_l=2\sum_{l\leq i^*,\ a_{l-1}\geq d_0}a_l\leq2S_{i^*}$.
\end{proof}

Two consequences make every constant below independent of $u$.  First, apply
the first inequality in \eqref{eq:workload} at $t=D_i$ for a retained index
$i$ and use
$V_u(D_i)\leq D_i$: since $i^*\geq i$ and $S_i\geq a_{i-1}+a_i$,
\[
 a_{i-1}+2(a_{i-1}+a_i)\leq Ca_{i-1},
\]
so the record factor $q_i=a_i/a_{i-1}$ satisfies
\begin{equation}\label{eq:factorbound}
  1\leq q_i\leq M:=\frac{C-3}{2}.
\end{equation}
In particular $C\geq5$.  Second, if $a_k$ is the first sorted magnitude at
least $d_0$, then the first inequality in \eqref{eq:workload} at $t=Cd_0$
gives
$d_0+2a_k\leq Cd_0$, hence the first retained deadline satisfies
\[
 D_{k+1}=Ca_k\leq T_0:=\frac{C(C-1)}2d_0,
\]
and all subsequent deadline ratios $D_{i+1}/D_i=q_i$ are at most $M$.

\section{Finite-window logarithmic scheduling}

The next lemma is purely one-dimensional.

\begin{lemma}[finite-window scheduling]\label{lem:scheduling}
Let $C>1$, $R=1-1/C$, and $t_0>0$.  Let $(D_i,J_i)_{i\in\mathcal J}$ be a
countable family of jobs with $J_i>0$, $D_i\geq t_0$, and only finitely many
$D_i$ in every bounded interval.  Let $V$ be absolutely continuous on compact
subintervals of $[t_0,\infty)$ and satisfy, for $t\geq t_0$,
\[
 0\leq V'\leq1\ \text{a.e.},\qquad V(t)\leq t,\qquad
 V(t)\geq\frac tC+\sum_{D_i\leq t}J_i .
\]
Suppose that, for some $0<\lambda<1$,
\[
  D_i-\frac{J_i}{R}\geq\lambda D_i
\]
for every job.  Then, for $t_0\leq L<U$,
\begin{equation}\label{eq:windowjobs}
 \int_L^U V'(t)\,\frac{dt}{t}
 \geq \frac1C\log\frac UL+
 \sum_{L/\lambda\leq D_i\leq U}
 R\log\frac{D_i}{D_i-J_i/R}-2R.
\end{equation}
\end{lemma}

\begin{proof}
\emph{Step 1: a monotone envelope.}
Set $z(t)=V(t)-t/C$, so that $-1/C\leq z'\leq R$ a.e., $z\geq0$ on
$[t_0,\infty)$, and $z(t)\geq\sum_{D_i\leq t}J_i$.  On $[L,U]$ take the
future running minimum
\[
 \bar z(t)=\min_{t\leq s\leq U}z(s),
 \qquad
 \bar V(t)=\frac tC+\bar z(t).
\]
Clearly $0\leq\bar z\leq z$, $\bar z$ is nondecreasing, and
$\bar z(U)=z(U)$.  It is $R$-Lipschitz: for $t<t'$, if the minimum defining
$\bar z(t)$ is attained at some $s\in[t,t']$ then
$\bar z(t')\leq z(t')\leq z(s)+R(t'-s)\leq\bar z(t)+R(t'-t)$, and otherwise
$\bar z(t')=\bar z(t)$.  For a deadline $D_i\in[L,U]$ and every
$s\in[D_i,U]$ we have $z(s)\geq\sum_{D_j\leq s}J_j\geq\sum_{D_j\leq D_i}J_j$,
so
\begin{equation}\label{eq:envelopejobs}
 \bar z(D_i)\geq\sum_{D_j\leq D_i}J_j\qquad(D_i\in[L,U]),
\end{equation}
and likewise $\bar z(L)\geq\sum_{D_j\leq L}J_j$.  Also
$\bar z(L)\leq z(L)\leq L-L/C=RL$.  Since $z-\bar z$ is absolutely
continuous, nonnegative and vanishes at $U$, integration by parts gives
\[
 \int_L^U(V'-\bar V')\,\frac{dt}{t}
 =\int_L^U(z'-\bar z')\,\frac{dt}{t}
 =-\frac{z(L)-\bar z(L)}L
   +\int_L^U\frac{z(t)-\bar z(t)}{t^2}\,dt
 \geq-\frac{z(L)}{L}\geq-R,
\]
and therefore
\begin{equation}\label{eq:step1}
 \int_L^U V'\,\frac{dt}{t}
 \geq\frac1C\log\frac UL+\int_L^U\bar z'(t)\,\frac{dt}{t}-R.
\end{equation}

\emph{Step 2: earliest-deadline allocation.}
Let $\mu$ be the measure on $[L,U]$ with an atom of mass $B=\bar z(L)$ at
$L$ and density $\bar z'$ on $(L,U]$; thus $\mu([L,t])=\bar z(t)$.  List
the finitely many jobs with $D_i\leq U$ in nondecreasing order of deadline
and let $A_n$ be the cumulative volume of the first $n$ of them.  Assign to
the $n$-th job the mass of $\mu$ in the level band $(A_{n-1},A_n]$, that is,
the part of the atom of size
$b_n=|(A_{n-1},A_n]\cap[0,B]|$ together with $\bar z'\,dt$ restricted to
$E_n=\{t\in(L,U]:A_{n-1}<\bar z(t)\leq A_n\}$.  Since $\bar z$ is
nondecreasing and absolutely continuous, for
$B\leq a<b\leq\bar z(U)$ the truncation $\min\{\max\{\bar z,a\},b\}$ is
absolutely continuous with derivative $\bar z'\mathbf 1_{\{a<\bar z<b\}}$
almost everywhere (a monotone absolutely continuous function has zero
derivative a.e.\ on each of its level sets), and integrating it over $[L,U]$
gives
\[
 \int_{\{a<\bar z(t)\leq b\}}\bar z'(t)\,dt=b-a.
\]
Thus the density part of $\mu$ pushed forward by $\bar z$ is Lebesgue
measure on $(B,\bar z(U)]$, and
$\int_{E_n}\bar z'\,dt=J_{i_n}-b_n$.  The sets $E_n$ are pairwise disjoint
and $\sum_nb_n\leq B$.  Jobs with deadline at most $L$ have cumulative
volume at most $B$ by \eqref{eq:envelopejobs}, so their level bands lie
entirely in the initial atom and their density allocation is zero.  For a
job with $D_{i_n}\in(L,U]$, \eqref{eq:envelopejobs} gives
$\bar z(D_{i_n})\geq A_n$.  Monotonicity then implies that
$E_n\cap(D_{i_n},U]$ is contained in the level set
$\{\bar z=A_n\}$, on which $\bar z'=0$ almost everywhere.  Hence the
density mass of every job sits in $(L,D_{i_n}]$.  In words: every job is
served before its deadline by the earliest available capacity, using $b_n$
units of the initial inventory.  In the next step we write $E_i$ and $b_i$
for the set and atom share belonging to the job whose original index is
$i=i_n$.

\emph{Step 3: cost of the selected jobs.}
Let $S=\{i:L/\lambda\leq D_i\leq U\}$.  Because the sets $E_n$ are disjoint
and $\bar z'\geq0$,
\[
 \int_L^U\bar z'\,\frac{dt}{t}\geq\sum_{i\in S}\int_{E_i}\bar z'\,\frac{dt}{t}.
\]

For $i\in S$ the density $\bar z'$ is at most $R$, its integral over
$E_i$ equals its integral over $E_i\cap(L,D_i]$ and is $J_i-b_i$.  Hence,
with $A=J_i-b_i$,
\[
 A\leq R(D_i-L),\qquad
 s:=D_i-\frac AR\geq L,
\]
and also $s\geq D_i-J_i/R\geq\lambda D_i>0$.  Since $1/t$ is decreasing,
the integral is at least that of the latest possible schedule: rate $R$ on
the interval $[s,D_i]\subseteq[L,D_i]$.  Indeed, subtracting the rate-$R$
indicator of $[s,D_i]$ from $\bar z'\mathbf 1_{E_i}$ gives a function of
integral zero which is nonnegative before $s$ and nonpositive after $s$;
multiplication by the decreasing function $1/t$ can only make its integral
nonnegative.  Therefore
\[
 \int_{E_i}\bar z'\,\frac{dt}{t}
 \geq R\log\frac{D_i}{D_i-(J_i-b_i)/R}
 = R\log\frac{D_i}{D_i-J_i/R}
   -R\log\Bigl(1+\frac{b_i}{R(D_i-J_i/R)}\Bigr).
\]
By $\log(1+x)\leq x$ and $D_i-J_i/R\geq\lambda D_i\geq L$ for $i\in S$,
the prepayment loss is at most $b_i/L$, and summing over $S$ loses at most
$B/L\leq R$.  Combining with \eqref{eq:step1} proves \eqref{eq:windowjobs}.
\end{proof}

Apply the lemma to \eqref{eq:workload} with $t_0=Cd_0$: the retained jobs
\eqref{eq:jobs} have $D_i\geq Cd_0$, and finitely many deadlines lie in any
bounded interval by Lemma~\ref{lem:records}(vii).  Take
\[
 R=\frac{C-1}{C},
 \qquad
 \lambda=\frac2{C-1}\in(0,1).
\]
Indeed,~\eqref{eq:factorbound} gives
\[
 D_i-\frac{J_i}{R}
 =D_i\left(1-\frac{2q_i}{C-1}\right)
 \geq\lambda D_i.
\]
Define
\begin{align}
 h_C(q)&=-\log\left(1-\frac{2q}{C-1}\right),\label{eq:hdef}\\
 m(C)&=\inf_{1<q<(C-1)/2}\frac{h_C(q)}{\log q},\label{eq:mdef}\\
 G(C)&=\frac1C+\left(1-\frac1C\right)m(C).\label{eq:Gdef}
\end{align}
Then $m(C)\geq0$, and $h_C(q_i)\geq m(C)\log q_i$ for every retained job:
for $q_i>1$ by definition, and for $q_i=1$ because $h_C(1)>0$.
Figure~\ref{fig:schedule} shows the schedule for which
Lemma~\ref{lem:scheduling} is tight when all record factors equal the
minimizer in \eqref{eq:mdef}.

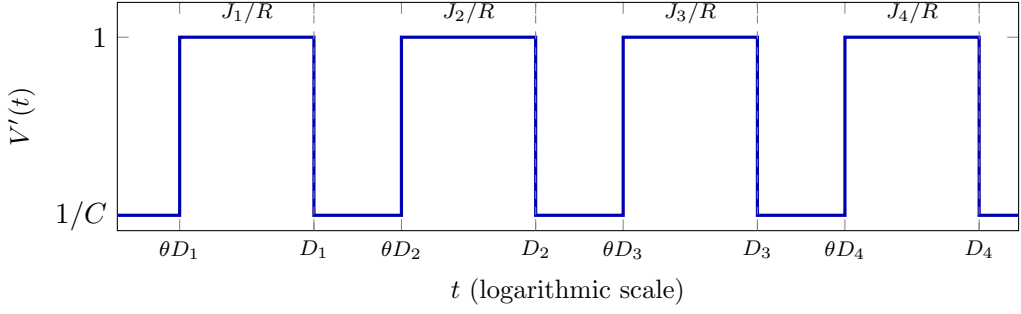
\begin{figure}[t]
\centering
\begin{tikzpicture}
\begin{axis}[width=0.9\textwidth,height=4.6cm,xmode=log,xmin=0.5,xmax=12,
  ymin=0,ymax=1.18,ytick={0.0794,1},yticklabels={$1/C$,$1$},
  xtick={0.6229,1,1.3616,2.1857,2.9760,4.7774,6.5049,10.4420},
  xticklabels={$\theta D_1$,$D_1$,$\theta D_2$,$D_2$,$\theta D_3$,$D_3$,$\theta D_4$,$D_4$},
  xticklabel style={font=\scriptsize},ylabel={$V'(t)$},xlabel={$t$ (logarithmic scale)},
  xlabel style={font=\small},ylabel style={font=\small},grid=none]
\addplot[const plot,very thick,blue!70!black] coordinates {
 (0.5,0.0794) (0.6229,1) (1,0.0794) (1.3616,1) (2.1857,0.0794) (2.9760,1) (4.7774,0.0794) (6.5049,1) (10.4420,0.0794) (12,0.0794)};
\foreach \d in {1,2.1857,4.7774,10.4420} { \addplot[dashed,gray] coordinates {(\d,0) (\d,1.18)}; }
\node[font=\scriptsize,anchor=south] at (axis cs:0.789,1.02) {$J_1/R$};
\node[font=\scriptsize,anchor=south] at (axis cs:1.725,1.02) {$J_2/R$};
\node[font=\scriptsize,anchor=south] at (axis cs:3.77,1.02) {$J_3/R$};
\node[font=\scriptsize,anchor=south] at (axis cs:8.24,1.02) {$J_4/R$};
\end{axis}
\end{tikzpicture}
\caption{The extremal per-direction schedule at $C=C_{\log}$ and record
factor $q=q_{\log}$: geometric deadlines $D_{i+1}=qD_i$, full speed on the
latest interval $[\theta D_i,D_i]$ of length $J_i/R$ before each deadline,
where $\theta=1-2q/(C-1)=0.623\ldots$, and the baseline rate $1/C$
otherwise.  The full-speed intervals occupy the fraction
$h_C(q)/\log q=m(C)=0.605\ldots$ of logarithmic time, so the log-time
average of $V'$ is $1/C+(1-1/C)m(C)=G(C)=2/\pi$: this is the profile that
Lemma~\ref{lem:scheduling} shows to be cheapest, and that the angular budget
of Section~\ref{sec:angular} forbids to be beaten on average over
directions.}
\label{fig:schedule}
\end{figure}

Choose $L$ so large that $L/\lambda>T_0$.  Let $D_m$ and $D_n$ be the first
and last deadlines in $[L/\lambda,U]$ (if there is none the estimate below
is trivial, see the end of this section).  There is a preceding retained
deadline $D_{m-1}<L/\lambda$, and~\eqref{eq:factorbound} gives
\[
 D_m\leq\frac{M}{\lambda}L,
 \qquad D_{n+1}>U.
\]
Since $D_{i+1}/D_i=q_i$,
\begin{equation}\label{eq:logspan}
 \sum_{L/\lambda\leq D_i\leq U}\log q_i
 =\log\frac{D_{n+1}}{D_m}
 \geq\log\frac UL-\log\frac M\lambda.
\end{equation}
Combining Lemma~\ref{lem:scheduling}, the job penalty
$R h_C(q_i)\geq Rm(C)\log q_i$, and~\eqref{eq:logspan} proves the uniform
directional estimate
\begin{equation}\label{eq:uniformactivity}
 \int_L^U|\langle u,\gamma'(t)\rangle|\,\frac{dt}{t}
 \geq G(C)\log\frac UL-B_C,
\end{equation}
where
\begin{equation}\label{eq:Bdef}
 B_C=2R+Rm(C)\log\frac M\lambda
 =2R+Rm(C)\log\left(\frac{C-1}{2}\cdot\frac{C-3}{2}\right).
\end{equation}
If no deadline lies in $[L/\lambda,U]$ then $U<D_m\leq ML/\lambda$, so the
right-hand side of \eqref{eq:uniformactivity} is at most
$\frac1C\log\frac UL-2R$, which is a lower bound for the left-hand side by
\eqref{eq:windowjobs}.  Crucially, $B_C$ and the admissible lower threshold
for $L$ depend only on $C$ and $d_0$, not on $u$.

\section{The angular budget and the lower bound}\label{sec:angular}

We now add the directional inequalities without losing a geometric constant.

\begin{proof}[Proof of Theorem~\ref{thm:main}]
Let $\gamma$ be $C$-competitive; by \eqref{eq:factorbound}, $C\geq5$, so
$\lambda<1$ and the domain of the infimum \eqref{eq:mdef} is nonempty.  For
$u_\theta=(\cos\theta,\sin\theta)$, Tonelli's theorem and unit-speed
parametrization give
\begin{align*}
 \int_0^\pi\int_L^U
 |\langle u_\theta,\gamma'(t)\rangle|\,\frac{dt}{t}\,d\theta
 &=\int_L^U\left(\int_0^\pi
 |\langle u_\theta,\gamma'(t)\rangle|\,d\theta\right)\frac{dt}{t}\\
 &=2\log\frac UL.
\end{align*}
Integrating~\eqref{eq:uniformactivity} over $\theta$ therefore yields
\[
 2\log\frac UL
 \geq\pi G(C)\log\frac UL-\pi B_C.
\]
Keep $L$ fixed above its uniform threshold and let $U/L\to\infty$.  We obtain
\begin{equation}\label{eq:angularconstraint}
  G(C)\leq\frac2\pi.
\end{equation}

It remains to identify the threshold.  Put $k=2/(C-1)$, $x=\log q$, and
\[
 \mathcal L_C(x)=-\log(1-ke^x),\qquad 0<x<\log(1/k).
\]
The derivative numerator of $\mathcal L_C(x)/x$ is
\[
 \mathcal H_C(x)=x\mathcal L_C'(x)-\mathcal L_C(x),
 \qquad
 \mathcal H_C'(x)=x\mathcal L_C''(x)
 =\frac{xke^x}{(1-ke^x)^2}>0.
\]
It tends to $\log(1-k)<0$ as $x\downarrow0$ and to $+\infty$ at the other
end of the domain.  Thus the infimum in~\eqref{eq:mdef} is attained at one
unique record factor $q_C=e^{x_C}$.  Since
$\partial_x\mathcal H_C(x_C)=\mathcal H_C'(x_C)>0$, the implicit-function
theorem makes $x_C$, and hence $q_C$ and $G(C)$, continuous (indeed smooth)
in $C$ on the range under consideration.

For fixed admissible $q$, the expression
\[
 f_q(C)=\frac1C+\left(1-\frac1C\right)\frac{h_C(q)}{\log q}
\]
is strictly decreasing in $C$: writing $\ell=\log q$ and $z=2q/(C-1)$,
\[
 C^2\ell\,f_q'(C)=h_C(q)-\ell-\frac{2qC}{C-1-2q}<0
\]
because $h_C(q)=-\log(1-z)<z/(1-z)=2q/(C-1-2q)$.  If $C_1<C_2$ and $q_1$
is the minimizer at $C_1$, which is admissible at $C_2$, then
$G(C_1)=f_{q_1}(C_1)>f_{q_1}(C_2)\geq G(C_2)$; so $G$ is strictly
decreasing.  Consequently there is at most one solution of
\[
  G(C_{\log})=\frac2\pi.
\]
Continuity together with the opposite certified signs in
Section~\ref{sec:numerics} proves existence and, in particular,
$C_{\log}>12.5937$.
If $C<C_{\log}$, strict monotonicity gives $G(C)>2/\pi$, contradicting
\eqref{eq:angularconstraint}.  Hence every competitive constant is at least
$C_{\log}$.
\end{proof}

\section{Three consistency checks}\label{sec:checks}

The argument uses the plane only through the disk deadlines and the angular
budget $\int_0^\pi|\langle u_\theta,v\rangle|\,d\theta=2\|v\|$.  Replacing
the budget by the one appropriate to a restricted family of target lines
gives lower bounds for the corresponding restricted problems, which can be
compared with known values.  The script \path{scripts/sanity_checks.py}
reproduces the numbers below.  Figure~\ref{fig:G} summarizes them.

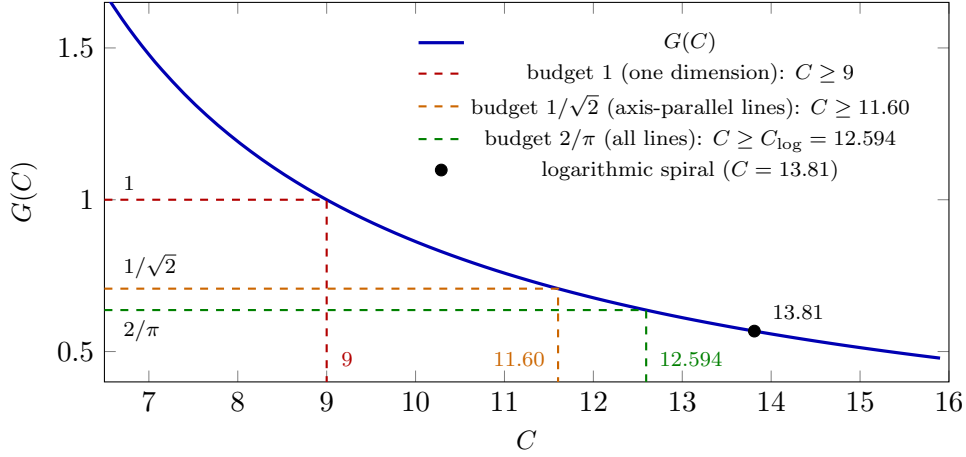
\begin{figure}[t]
\centering
\begin{tikzpicture}
\begin{axis}[width=0.85\textwidth,height=6.6cm,xmin=6.5,xmax=16,ymin=0.4,ymax=1.65,
  xlabel={$C$},ylabel={$G(C)$},xlabel style={font=\small},ylabel style={font=\small},
  xtick={7,8,9,10,11,12,13,14,15,16},
  ytick={0.5,1,1.5},
  legend style={font=\scriptsize,at={(0.97,0.95)},anchor=north east,draw=none}]
\addplot[very thick,blue!70!black,smooth] coordinates {
 (6.550,1.6599) (6.700,1.5945) (6.850,1.5342) (7.000,1.4784) (7.150,1.4267) (7.400,1.3482) (7.650,1.2782) (7.900,1.2153) (8.150,1.1584) (8.400,1.1067) (8.650,1.0595) (8.900,1.0163) (9.150,0.9765) (9.400,0.9398) (9.650,0.9058) (9.900,0.8742) (10.150,0.8448) (10.400,0.8174) (10.650,0.7917) (10.900,0.7675) (11.150,0.7449) (11.400,0.7235) (11.650,0.7034) (11.900,0.6844) (12.150,0.6663) (12.400,0.6493) (12.650,0.6330) (12.900,0.6176) (13.150,0.6029) (13.400,0.5890) (13.650,0.5756) (13.900,0.5629) (14.150,0.5507) (14.400,0.5390) (14.650,0.5278) (14.900,0.5171) (15.150,0.5068) (15.400,0.4969) (15.650,0.4874) (15.900,0.4783)};
\addlegendentry{$G(C)$}
\addplot[dashed,red!70!black,thick] coordinates {(6.5,1) (9,1) (9,0.4)};
\addlegendentry{budget $1$ (one dimension): $C\geq9$}
\addplot[dashed,orange!80!black,thick] coordinates {(6.5,0.70711) (11.603,0.70711) (11.603,0.4)};
\addlegendentry{budget $1/\sqrt2$ (axis-parallel lines): $C\geq11.60$}
\addplot[dashed,green!50!black,thick] coordinates {(6.5,0.63662) (12.594,0.63662) (12.594,0.4)};
\addlegendentry{budget $2/\pi$ (all lines): $C\geq C_{\log}=12.594$}
\addplot[only marks,mark=*,mark size=2.2pt,black] coordinates {(13.811,0.5673)};
\addlegendentry{logarithmic spiral ($C=13.81$)}
\node[font=\scriptsize,anchor=south west] at (axis cs:6.6,1.0) {$1$};
\node[font=\scriptsize,anchor=south west] at (axis cs:6.6,0.7071) {$1/\sqrt2$};
\node[font=\scriptsize,anchor=north west] at (axis cs:6.6,0.6366) {$2/\pi$};
\node[font=\scriptsize,anchor=south west,red!70!black] at (axis cs:9.05,0.42) {$9$};
\node[font=\scriptsize,anchor=south east,orange!80!black] at (axis cs:11.57,0.42) {$11.60$};
\node[font=\scriptsize,anchor=south west,green!50!black] at (axis cs:12.63,0.42) {$12.594$};
\node[font=\scriptsize,anchor=south west] at (axis cs:13.9,0.575) {$13.81$};
\end{axis}
\end{tikzpicture}
\caption{The strictly decreasing function $G$ of \eqref{eq:Gdef}.  A
family of target lines imposes a velocity budget $\beta$ (the maximal
log-time average of the projected speed summed over the relevant
directions), and the argument gives $G(C)\leq\beta$; the three budgets
$1$, $1/\sqrt2$, $2/\pi$ yield the bounds $9$, $11.60$, $12.594$.  The
logarithmic spiral lies to the right of all of them.}
\label{fig:G}
\end{figure}

\begin{corollary}[one dimension]\label{cor:cowpath}
For a scalar search ($u$ fixed, $V=$ variation of the path itself, so
$V'\leq1$), \eqref{eq:uniformactivity} and $V'\leq1$ give $G(C)\leq1$, and
$G(C)\leq1$ holds if and only if $C\geq9$.
\end{corollary}

\begin{proof}
Recall $C\geq5$ by \eqref{eq:factorbound}, and that the infimum defining
$m(C)$ is attained.  Then $G(C)\leq1$ iff $m(C)\leq1$ iff
$h_C(q)\leq\log q$ for some admissible $q$ iff
$\sup_{1<q<(C-1)/2} q\bigl(1-\frac{2q}{C-1}\bigr)\geq1$.  For $C>5$ the
supremum is attained at the admissible point $q=(C-1)/4$ and equals
$(C-1)/8$, while for $C=5$ it is $\tfrac12$; so the condition is $C\geq9$.
At $C=9$
the extremal factor is $q=2$, and indeed $h_9(q)-\log q=-\log(q-q^2/4)\geq0$
with equality only at $q=2$.
\end{proof}

Thus the scheduling bound reproduces the Beck--Newman constant
\eqref{eq:cowpath} exactly, together with the doubling factor of the optimal
cow-path strategy.  In this sense Lemma~\ref{lem:scheduling} is tight in one
dimension.

\begin{remark}[axis-parallel shoreline]
If only lines parallel to the coordinate axes are to be found, the two
coordinate projections satisfy the same deadlines, and
$|x'|+|y'|\leq\sqrt2$ replaces the angular identity.  The argument then
gives $2G(C)\leq\sqrt2$, i.e.\ $C\geq11.6027\ldots$.  This is below
$12.5385\ldots$, the value of Langetepe's strategy for that problem
\cite{langetepe2012} (which improved the earlier $12.5406$ of
\cite{jez2009}), as it must be; the difference measures what the two-axis
budget loses compared with the exact angular average.
\end{remark}

\begin{remark}[the logarithmic spiral]
For the spiral $\gamma(\varphi)=e^{b\varphi}(\cos\varphi,\sin\varphi)$ with
the optimal $b=0.21247\ldots$ one recomputes $C=13.8111\ldots$; each
projection has geometric record factor $q=e^{b\pi}=1.9493\ldots$, so the
sorted record sequence coincides with the record sequence.  Numerically,
the workload inequality \eqref{eq:workload} holds for all $t$, with equality
at every deadline $t=D_i$ (slack $<10^{-12}$ there, and $0.127\,t$ midway
between deadlines): the spiral's worst-case lines are exactly the ones that
define the deadlines, so it is extremal for the workload constraint.  The
window inequality \eqref{eq:windowjobs} holds on the windows tested, and
the asymptotic per-direction activity satisfies
$2/\pi=0.6366\ldots\geq 1/C+Rh_C(q)/\log q=0.5767\ldots\geq
G(C)=0.5673\ldots$.  The extremal record factor of the theorem,
$q_{\log}=2.1857\ldots$, differs from the spiral's $1.9493\ldots$; the gap
between $12.5937$ and $13.8111$ reflects that the extremal per-direction
schedules of Lemma~\ref{lem:scheduling} cannot be realized simultaneously in
all directions by one planar path.
\end{remark}

\section{Certified numerical enclosure}\label{sec:numerics}

Only the final one-dimensional constant is numerical.  The script
\path{scripts/certify_constant.py} uses Arb balls at $256$-bit precision.
For each of the rational decimal endpoints
\[
 C_-=12.59370967012466,
 \qquad
 C_+=12.59370967012468,
\]
it brackets the unique zero of the strictly increasing stationarity function
\[
 H(C,q)=\frac{2q\log q}{C-1-2q}
        +\log\left(1-\frac{2q}{C-1}\right)
\]
(this is $\mathcal H_C(\log q)$ in the notation of the previous section).
Evaluating the activity on the entire minimizing-factor bracket gives
\begin{align*}
 G(C_-)-\frac2\pi
 &\in 4.7852509208884\times10^{-16}+[-9,9]\times10^{-48},\\
 G(C_+)-\frac2\pi
 &\in-8.0053211211654\times10^{-16}+[-2,2]\times10^{-47}.
\end{align*}
Both signs are therefore rigorous.  The associated record factor is
\[
 q_{\log}=2.185730264988\ldots.
\]

The repository also contains small Lean/mathlib files checking the exact
crossing lemma, the one-Lipschitz projection reduction, the finite
sum-splitting inequality at the heart of Lemma~\ref{lem:normalform}, the
deadline factor bound, and the logarithmic prepayment inequality.  These
checks support the algebra used above; the analytic compactness and
integration arguments remain the paper proof rather than a claim of complete
formalization.

\section*{Disclosure of AI assistance}

Large language models were used in the preparation of this work: OpenAI
GPT~5.6 Sol and Anthropic Claude Fable~5 assisted with numerical
experiments, exploration of proof strategies, and editing of the text.  All
definitions, statements, and proofs were checked by the author, who takes
full responsibility for the content.

\appendix

\section{Proof of Lemma~\ref{lem:records}}\label{app:turns}

Throughout, $x$ satisfies the standing assumptions of
Section~\ref{sec:normalform}.  We use one observation repeatedly:

\smallskip\noindent\emph{Limit claim.}  If $a_n\uparrow a$ are record times
of one sign, say $a_n\in P$, and $a>t_{\mathrm{dep}}$, then $a\in P$.
Indeed $x(a)=M(a)$ by continuity, and if $M(s)=M(a)$ for some $s<a$, then
for large $n$ we would have $s<a_n$ and $M(a_n)=M(a)=M(s)$, contradicting
that $a_n$ is a strict record.

\smallskip\noindent(i)
Fix $t>t_{\mathrm{dep}}$.  Since $x$ is not identically zero on $[0,t]$,
$M(t)>0$ or $N(t)>0$, and the first passage time of the level $M(t)$ or
$-N(t)$ is a record time in $(t_{\mathrm{dep}},t]$.  Let $a$ be the
supremum of the record times in $(t_{\mathrm{dep}},t]$; we show that it is
attained.  If not, choose record times $a_n\uparrow a$ with $a_n<a$.  Were
both signs to occur infinitely often, continuity along the two subsequences
would give $x(a)=M(a)=-N(a)$, hence $M(a)=N(a)=0$ and $a\leq
t_{\mathrm{dep}}$, which is absurd; so one sign occurs infinitely often, and
the limit claim shows that $a$ is a record time.  Thus $\rho(t)=a$ exists.
Next, $\rho(t)$ has the same sign as $t$ because $\rho(\rho(t))=\rho(t)$,
and no record time lies in $(\rho(t),t]$, so $\rho$ and hence the sign are
constant on $[\rho(t),t]$; by maximality this interval lies in the phase of
$t$, which therefore contains the record time $\rho(t)$.  The existence of a
last record time in each phase is proved after (ii).

\smallskip\noindent(ii)
Fix $t_{\mathrm{dep}}<\delta<T$ and put $\eta=M(\delta)+N(\delta)>0$.  Every
phase that meets $[\delta,T]$, other than the phases containing $\delta$ and
$T$, is an interval contained in $(\delta,T)$ and hence, by (i), contains a
record time in $(\delta,T)$.  A positive record time in $(\delta,T)$ has
value at least $M(\delta)$ and a negative one has magnitude at least
$N(\delta)$.  Between two distinct phases of the same sign lies a phase of
the opposite sign (otherwise maximality would merge them), hence a record
time of the opposite sign.  Consequently, if we pick one record time in
$(\delta,T)$ from each of finitely many distinct such phases and list them in
temporal order, then between consecutive members the path travels from a
level $\geq M(\delta)$ to a level $\leq-N(\delta)$ or back, at variation
cost at least $\eta$ over disjoint time intervals.  Local bounded variation
bounds the size of every such family by $1+\Var(x;[\delta,T])/\eta$, so only
finitely many phases meet $[\delta,T]$.  Hence the phases are locally finite
in $(t_{\mathrm{dep}},\infty)$, countable, and temporally ordered like an
interval $K\subseteq\Z$.  Since $x$ reaches arbitrarily large levels of both
signs there are infinitely many phases after any given one, so $K$ is
unbounded above and every phase is bounded.  If $K$ is unbounded below,
local finiteness on every $[\delta,T]$ forces the phases, and their turn
times, to decrease to $t_{\mathrm{dep}}$.

\smallskip\noindent\emph{Last record time.}
Let $a$ be the supremum of the record times of a phase $J$; it is finite
because $J$ is bounded, and it is a record time of the sign of $J$ (either
attained, or by the limit claim).  It lies in $J$: it lies in the closure of
$J$ and has the sign of $J$, while if $a\notin J$ then $a=\sup J$ would
belong to the phase immediately following $J$, whose sign is opposite by
maximality.  So $a$ is the last record time of $J$.

\smallskip\noindent(iii) is maximality of phases, and $x(t_k)=\varepsilon_km_k$
because $t_k$ is a record time of sign $\varepsilon_k$.

\smallskip\noindent(iv)
During a negative phase no positive record time occurs, so $M$ is constant
there, and the turn magnitude of a positive phase is the value of $M$ at
the end of that phase; $M$ strictly increases from one positive phase to the
next because the later phase contains positive record times.  Unboundedness
follows since $x$ reaches every level.

\smallskip\noindent(v)
Between $t_{k-1}$ and $t_k$ the path travels from $\varepsilon_{k-1}m_{k-1}$
to $-\varepsilon_{k-1}m_k$.

\smallskip\noindent(vi)
$\sigma^{+}(r)$ is by definition a positive record time; the phase $j$
containing it is positive; $m_j=M(t_j)\geq M(\sigma^{+}(r))=r$; the phase
starts after the turn $t_{j-1}$ of the preceding phase, and $t_j$ is the last
record time of the phase; earlier positive turns are values of $M$ before
$\sigma^{+}(r)$, hence $<r$.  The negative case is symmetric.

\smallskip\noindent(vii)
A turn of magnitude $\leq X$ is the first passage time of a level of
magnitude $\leq X$, hence occurs no later than $\tau(X)<\infty$; by (v) each
turn of magnitude $\geq\varepsilon_0$ is preceded by a leg of variation
$\geq\varepsilon_0$, and these legs are disjoint, so there are at most
$\Var(x;[0,\tau(X)])/\varepsilon_0$ of them.  If $K$ is unbounded below,
(ii) gives $t_k\downarrow t_{\mathrm{dep}}$, and
$m_k\leq M(t_k)+N(t_k)\to0$.  For the last claim, if $K$ has a least element
prepend the departure point with magnitude zero and sum (v) over the disjoint
legs ending by $T$; if $K$ is unbounded below, perform the same sum from a
finite lower index and let that index tend to $-\infty$.

{\footnotesize
\bibliographystyle{abbrv}
\bibliography{references}
}

\end{document}